\documentclass[12pt]{article}

\usepackage{amsmath}
\usepackage{amsthm}
\usepackage{graphicx}
\usepackage{amsfonts}
\usepackage{fullpage}
\usepackage{ amssymb }
\usepackage{cite}
\usepackage{hyperref}

\title{Look, Knave}
	\author{Thomas Morrill\footnote{Supported by Australian Research Council Discovery Project DP160100932 }\\ School of Science\\ The University of New South Wales Canberra, Australia\\
	\href{mailto:t.morrill@adfa.edu.au}{\nolinkurl{t.morrill@adfa.edu.au}}}

\newtheorem{cor}{Corollary}
\newtheorem{thm}{Theorem}
\newtheorem{lemma}{Lemma}

\newcommand{\one}{\texttt{1}}
\newcommand{\zero}{\texttt{0}}
\newcommand{\ones}{\texttt{1}{s}}
\newcommand{\zeros}{\texttt{0}s}

\begin{document}

\maketitle

\section{Introduction}

\begin{abstract}
  We examine a recursive sequence in which $s_n$ is a literal description of what the binary expansion of the previous term $s_{n-1}$ is not.
  By adapting a technique of Conway, we determine limiting behaviour of $\{s_n\}$ and dynamics of a related self-map of $2^{\mathbb{N}}$.
  Our main result is the existence and uniqueness of a pair of binary sequences, each the compliment-description of the other.
  We also take every opportunity to make puns.
\end{abstract}

The Look-Say sequence is defined as follows.
Let $s_1 = 1$.
Given $s_n$, the next term of the sequence is a literal description of the digits of the previous \cite{Look}.
The first few terms are
\[
1, 11, 21, 1211, 111221, \ldots
\]
We'll use $|s|$ to denote the length of a finite string $s$.
\begin{thm}[Conway, \cite{Conway, Lambda}] \label{Conway}
  Let $s_n$ be the $n$th term of the Look-Say sequence.
  Then
  \[
    \lim_{n \to \infty} \frac{|s_{n+1}|}{|s_n|} = \lambda,
  \]
  where
  \[
    \lambda = 1.3035\ldots.
  \]
\end{thm}
Shockingly, $\lambda$ is an algebraic integer of degree $71$.
Theorem \ref{Conway} follows from Conway's Cosmological Theorem \cite{Conway}.
In short, the terms of any Look-Say type sequence (not necessarily starting at $s_1 = 1$) will eventually decompose into a concatenation of certain fundamental substrings identified by Conway, his ``elements''.

This problem has also been considered in terms of binary strings.
Given a binary string $s_n$, the next term of the Binary Look-Say sequence is a literal description of the bits of the previous term, where the counts are expressed in base $2$ \cite{Look2}.
The first few terms are
\[
1, 101, 101100101, \ldots
\]

\begin{thm}[Johnston, \cite{Johnston}] \label{Johnston}
  Let $s_n$ be the $n$th term of the Binary Look-Say sequence.
  Then
  \[
    \lim_{n \to \infty} \frac{|s_{n+1}|}{|s_n|} = \lambda,
  \]
  where
  \[
    \lambda = 1.465571\ldots.
  \]
\end{thm}

We'll shake this up by introducing a new player, a Knave in the style of Smullyan.
As opposed to the previous recursions, our $s_n$ is instead the literal description of what the bits of $s_{n-1}$ aren't.
Our main result concerns the limiting behaviour of the \emph{Look-Knave Sequence}.

\begin{thm}\label{me}
  There is a unique pair of binary sequences $S_{even}$ and $S_{odd}$ such that $S_{even}$ is a literal description of the bitwise compliment of $S_{odd}$, and vice versa.
\end{thm}

The rest of the paper is organised as follows.
In Section \ref{knave} we define the Look-Knave Sequence and pose our problem.
Then, in Section \ref{kafka}, we simplify the problem and prove Theorem \ref{me}.
Finally, in Section \ref{future}, we offer avenues for future work.

\section{The Knave} \label{knave}

Recall Smullyan's game of Knights and Knaves, a logic puzzle in which Knights always tell the truth, and Knaves are alway compelled to lie \cite{Knave}.
Our Knave is a very idiosyncratic liar.
When the Knave looks at a string of $n$ \zeros, they correctly tell us they see $n$ bits of the same parity, but they will lie by saying that there are $n$ \ones.
Likewise, while looking at $k$ \ones, the Knave will happily tell us there are $k$ \zeros{} instead.

The Knave understands how to express natural numbers in base $2$, and will write down their observations for us as such.
Thus, when the Knave looks at the string
\[
\texttt{110},
\]
they write down
\[
\texttt{10  0  1  1}
\]
for the two \zeros{}  and one \one{} they claim to have seen.
Here, we have inserted whitespace to enhance the Knave's handwriting.

Now, our Knave has not yet realized that they could have lied about their count by inverting the bits representing $n$ and $k$ above.
I won't tell them if you won't.

Thus begins our new game.
We will supply a binary string, and command ``Look, Knave''.
Dutifully, the Knave will read the string, then record their observations on a fresh piece of paper for us.
We return this paper to the Knave, who reads their own report and transcribes it in the only way they can.
The game continues.

Let's begin with the string $s_1 = \one$, and take $s_n$ to be the Knave's description of $s_{n-1}$.
This defines the Look-Knave sequence.
For example, $s_3 = \texttt{1011}$.
We see that there is one bit which is not \zero,
followed by one bit which is not \one,
then two bits which are not \zero.
Thus, $s_4$ must be the string \texttt{1011100}.
In short, $s_n$ is a a binary string describing precisely what $s_{n-1}$ is not.

\begin{table}
\begin{center}
\begin{tabular}{l l}
$s_{2n+1}$ & $s_{2n+2}$\\
\hline
1 & 10 \\
1011 & 1011100 \\
1011110101 & 1011100011101110 \\
10111101111101111011 & 1011100011101011100011100 \\
1011110111110111011110111110101 & 101110001110101111011100011101011101110
\end{tabular}
\caption{\label{table} The first ten entries of the Look-Knave Sequence.}
\end{center}
\end{table}

Looking at Table \ref{table}, it is tempting to conjecture that the subsequences $\{s_{2n+1}\}$ and $\{s_{2n+2}\}$ are approaching some bitwise limits.
So, do there exist binary sequences $S_{even}$ and $S_{odd}$ such that $S_{odd}$ is the Knave's description of $S_{even}$, and vice versa?

A binary sequence $S$ can be described by the Knave, so long as the tail end of $S$ is not all \zeros{} or all \ones.
Let $\mathcal{S} \subset 2^{\mathbb{N}}$ be the set of all such sequences.
Then the Knave imposes a map $k: \mathcal{S} \to \mathcal{S}$.

It will be convenient to view finite strings as belonging to $2^{\mathbb{N}}$.
We'll say that a string whose final bit is \zero{} is followed by a tail of all \ones, and vice versa.
For example,
\begin{align*}
  \texttt{101} &\leftrightarrow \texttt{101000\ldots} \\
  \texttt{100} &\leftrightarrow \texttt{100111\ldots}.
\end{align*}
Our Knave doesn't have the patience for these infinite matters,
so when we do compel them to act on $2^{\mathbb{N}}$, the Knave will report
\[
 \texttt{000\ldots}
\]
as
\[
  \texttt{111\ldots}
\]
and vice-versa.
Thus, these tails will never interfere with the preceding string.
We will (somewhat abusively) treat these either as sequences or strings, depending on which is more convenient.

Note that $k$ is not invertible; already
\[
 k(\texttt{10}) = k(\texttt{11111}) = 1011.
\]

\section{Metamorphosis} \label{kafka}

For a natural number $n$, let $[n]$ denote the string which represents $n$ in base $2$.
We'll call any string of $n$ \zeros{} or $k$ \ones{} a \emph{ribbit}, short for Repeated BIT.
If we need to clarify what bit is repeated, we can say that \texttt{111} is a \emph{ribbit of three \ones}, or an \emph{odd ribbit}.
Likewise, \texttt{000} is a \emph{ribbit of three \zeros}, and an \emph{even ribbit}.
Thus, any binary sequence $S \in \mathcal{S}$ decomposes into a sequence of ribbits of alternating parity.
Aristophanes would tell us the $\mathcal{S}$ stands for $\mathcal{S}$ongs.

Let $S \in \mathcal{S}$.
Since the Knave must begin their report with a \one, we'll assume that $S$ begins with an odd ribbit.
Then $S$ decomposes into ribbits as
\[
  S = r_1\ r_2\ r_3\ \ldots
\]
Hoppily, this means that odd ribbits are indexed by odd subscripts and vice-versa.

We may write
\[
  k(S) = [|r_1|] \ \zero \ [|r_2|] \ \one \  [| r_3|] \ \zero \ldots.
\]
It is unfortunate here that the \one{} arising from $r_{2\ell + 1}$ will always form a ribbit with $[|r_{2\ell + 2}|]$.
Further, this \one{} can form a ribbit with $[|r_{2\ell + 1}|]$ (or a \zero{} with $[|r_{2\ell}|]$), depending on the final bit of $[|r_{2\ell + 1}|]$ (resp., $[|r_{2\ell}|]$).
However, the decomposition of $s_n$ into even and odd ribbits allows us to get the Knave's reports piecemeal; keeping
\[
  S = r_1\ r_2\ r_3\ \ldots
\]
with $r_1$ odd, then
\[
  k(S) = k(r_1)\ k(r_2 \ r_3)\ k(r_4 \ r_5) \ldots. 
\]
Thus, we can determine the behaviour of $k$ by examining all possible pairs of ribbits occurring in the decomposition of all $s_n$.
Fortunately, there are not many to check.

\begin{lemma}
  Let $\{ s_n\}$ be the Look-Knave sequence.
  A maximal ribbit occurring in $s_n$ cannot have length greater than five.
\end{lemma}

\begin{proof}
  Suppose $n$ is the smallest index such that $s_n$ contains a ribbit $r$ of length six or greater, either
  \[
    s_n = \ldots \one \overbrace{\zero \ldots \zero}^{\geq 6} \one \ldots
  \]
  or
  \[
    s_n = \ldots \zero \overbrace{\one \ldots \one}^{\geq 6} \zero \ldots
  \]
  What is $s_n$ describing\footnote{Or rather, what \emph{isn't} $s_n$ describing?}?
  If $r$ is even, then $s_{n-1}$ contains an ribbit of length at least $64$; this ribbit can only occur if $s_{n-1}$ has a ribbit $r'$ such that the binary representation of $|r'|$ has at least five \zeros.
  This is a contradiction.
  
  The case where $r$ is odd is more complicated.
  We already see that such an $r$ could arise from an $r'$ in $s_{n-1}$, where the the binary representation of $|r'|$ has at least five \ones, which is again impossible.
  
  However, $r$ could represent the concatenation of two separate descriptions of ribbits; the first odd, and the second even.
  In this case,
  \[
  s_n = \ldots \ \overbrace{\ldots \zero \one \ldots \one} \ \one \ \overbrace{\one \ldots \one \zero \ldots}  \ \ldots,
  \]
  where the first overbrace indicates the binary expansion of the length of an odd ribbit in $s_{n-1}$, and the second overbrace indicates the binary expansion of the length of an even ribbit in $s_{n-1}$.
  From our assumption on $n$, we see that the only acceptable arrangement is
  \[
  s_n =  \ldots  \ \overbrace{\ldots \one \one \one} \ \one \ \overbrace{\one \one} \ \zero \  \ldots
  \]
  Unfortunately,
  \[
    \overbrace{\ldots \one \one \one} 
  \]
  is the binary expansion of some $n \geq 7$, and we croak.
\end{proof}

In fact, once we know the bound for maximal ribbits in general, we can tighten up the proof for even ribbits.

\begin{cor}
  A maximal even ribbit occurring in $s_n$ cannot have length greater than three.
\end{cor}

We may now examine the Knave's behaviour on all possible ribbit pairs $(r, r')$ occurring in some $s_n$.
This is shown in Table \ref{rr}.
Note that in all cases, $k(r\ r')$ is no shorter than $r r'$.

From our observation in Table \ref{table}, we want to determine if the sequences $\{s_{2n+1}\}$ and $\{s_{2n + 2}\}$ converge in $\mathcal{S}$.
To this end, we will endow  $2^{\mathbb{N}}$ with a simple metric.
Two distinct binary sequences $S, S'$ who first differ at the $n$th bit satisfy $d(S, S') = 2^{-n}$.
Note that $\mathcal{S}$ is not complete under this metric, but $2^{\mathbb{N}}$ is.

For $\ell \geq 1$, let $r_\ell$ be the string given by the first $\ell$ bits in $s_\ell$, extended to the of the last ribbit.
For example, $r_\ell$ is the string \texttt{1011}, taken from $s_3 = \texttt{1011}$.
  
\begin{lemma}
  For $\ell \geq 1$, the strings $s_{\ell + 1}$ and $s_{\ell + 3}$ agree up to the $|r_\ell|$th bit.
\end{lemma}

\begin{proof}
  Because $r_\ell$ begins with \texttt{10}, we see that $|k(r_\ell)| > |r_\ell|$.
  In the induction, we see that the first $|r_{\ell}|$ bits of $s_{\ell}$ and $s_{\ell + 2}$ determine at least the first $|r_{\ell}| +1$ bits of $s_{\ell + 1}$ and $s_{\ell + 3}$.
  \end{proof}

\begin{cor}
  The sequences $\{k^{2n}(\one)\}$ and $\{k^{2n}(\texttt{10})\}$ converge in $\mathcal{S}$.
\end{cor}

%

Thus, we can take $S_{even} = \lim_{n \to \infty} k^{2n}(\texttt{10})$ and  $S_{odd} = \lim_{n \to \infty} k^{2n}(\texttt{1})$.
It turns out, not only are $S_{even}$ and $S_{odd}$ fixed points of $k^2$, they attract all other orbits under $k$ in $2^{\mathbb{N}}$.

\begin{thm}
  Let $S \in \mathcal{S}$ be a binary sequence.
  Then either
  \[
    \lim_{n \to \infty} d(k^n(S), k^n(\one)) = 0.
  \]
  or
    \[
    \lim_{n \to \infty} d(k^n(S), k^{n}(\texttt{10})) = 0.
  \]
\end{thm}

\begin{proof}
  We claim that some iterate $k^n(S)$ begins with the substring \texttt{10}.
  Certainly, $k(S)$ begins with an odd ribbit, so we may assume $S$ also begins with a \one{} without loss of generality.
  Note that if $k(S)$ begins with an odd ribbit of length $\ell \geq 2$, then $k^2(S)$ begins with an odd ribbit of length strictly less than $\ell$.
  Otherwise, $k(S)$ begins with $\texttt{10}$, and so does $k^2(S)$.
  
  Assume without loss of generality that $S$ begins with \texttt{10}.
  Then $k(S)$ begins with \texttt{101}.
  Using the same argument above, we see that some iterate of $S$ begins with either \texttt{1 0 1 1 1 10} or \texttt{1 0 1 1 10}.
  At this point, the iterates $k^n(S)$ begin to metamorphose into either $S_{even}$ or $S_{odd}$, and any discrepancies are pushed out to the tail.
  \end{proof}

\begin{cor}
  Let $S$ be any binary sequence in $\mathcal{S}$.
  Then $\lim_{n \to \infty} k^{2n}(S)$ exists, and is equal to one of $S_{even}$ or $S_{odd}$.
\end{cor}

\begin{cor}
    The only fixed points of $k^2$ in $\mathcal{S}$ are $S_{even}$ and $S_{odd}$.
\end{cor}

\begin{table}
\begin{center}
\begin{tabular}{l|l}
$r \ r'$ & $k(r \ r')$ \\
\hline
\one & \texttt{10} \\
\texttt{01} & \texttt{1110} \\
\texttt{001} & \texttt{10110} \\
\texttt{0001} & \texttt{11110} \\
\texttt{011} & \texttt{11100} \\
\texttt{0011} & \texttt{101100} \\
\texttt{00011} & \texttt{111100} \\
\texttt{0111} & \texttt{11110} \\
\texttt{00111} & \texttt{101110} \\
\texttt{000111} & \texttt{111110} \\
\texttt{01111} & \texttt{111000} \\
\texttt{001111} & \texttt{1011000} \\
\texttt{0001111} & \texttt{1111000} \\
\texttt{011111} & \texttt{111010} \\
\texttt{0011111} & \texttt{1011010} \\
\texttt{00011111} & \texttt{1111010} \\
\end{tabular}
\end{center}
\caption{Elements of the Knave map. \label{rr}}
\end{table}

\section{Future Study}\label{future}

We have left open the question of the asymptotic growth of $|s_n|$.
Experimentally, we expect that
\[
  \lim_{n \to \infty} \frac{|s_{n+1}|}{|s_n|} = 1.12\ldots
\]
Adapting Johnston's argument to this problem would be an appropriate problem for an undergraduate student.

Further, we conjecture that that binary strings $S$ are in fact the sections of a larger dynamical system via the diagonal entries of certain Kermitian matrices.


\begin{thebibliography}{1}

\bibitem{Conway}
J.~H. Conway.
\newblock {\em The Weird and Wonderful Chemistry of Audioactive Decay}, pages
  173--188.
\newblock Springer New York, New York, NY, 1987.

\bibitem{Look}
OEIS~Foundation Inc.
\newblock The {O}n-{L}ine {E}ncyclopedia of {I}nteger {S}equences.
\newblock \url{https://oeis.org/A005150}, 2020.
\newblock Accessed: 2020-04-14.

\bibitem{Lambda}
OEIS~Foundation Inc.
\newblock The {O}n-{L}ine {E}ncyclopedia of {I}nteger {S}equences.
\newblock \url{https://oeis.org/A014715}, 2020.
\newblock Accessed: 2020-04-14.

\bibitem{Look2}
OEIS~Foundation Inc.
\newblock The {O}n-{L}ine {E}ncyclopedia of {I}nteger {S}equences.
\newblock \url{https://oeis.org/A001387}, 2020.
\newblock Accessed: 2020-04-14.

\bibitem{Johnston}
Nathaniel Johnston.
\newblock The binary ?look-and-say? sequence.
\newblock
  \url{http://www.njohnston.ca/2010/11/the-binary-look-and-say-sequence/},
  2010.
\newblock Accessed: 2020-04-14.

\bibitem{Knave}
Raymond~M Smullyan.
\newblock What is the name of this book? the riddle of dracula and other
  logical puzzles.
\newblock 1980.

\end{thebibliography}

\end{document}